\documentclass[12pt]{amsart}
\usepackage{amssymb}
\usepackage{mathrsfs}
\usepackage{amssymb}
\usepackage{amsfonts}
\usepackage[english]{babel}
\usepackage[T1]{fontenc}
\usepackage[latin1]{inputenc}
\usepackage{fullpage}
\usepackage{amsmath}
\usepackage[all]{xy}
\usepackage{stmaryrd}

\newtheorem{theorem}{Theorem}[section]
\newtheorem{lemma}[theorem]{Lemma}
\newtheorem{corollary}[theorem]{Corollary}
\theoremstyle{definition}
\newtheorem{definition}[theorem]{Definition}

\newtheorem{proposition}[theorem]{Proposition}

\theoremstyle{remark}
\newtheorem{remark}[theorem]{Remark}

\numberwithin{equation}{section}

\begin{document}

\title{A sphere theorem for Bach-flat manifolds with positive constant scalar curvature}


\author{Yi Fang}
\address{(Yi Fang) Department of Applied Mathematics, Anhui University of Technology, Ma'anshan, Anhui 243002, China}
\email{flxy85@163.com}

\author{Wei Yuan}
\address{(Wei Yuan) Department of Mathematics, Sun Yat-sen University, Guangzhou, Guangdong 510275, China}
\email{gnr-x@163.com}




\keywords{Bach-flat, sphere theorem, constant scalar curvature, gap theorem}

\thanks{This work was supported by NSFC (Grant No. 11521101, No. 11601531), \emph{The Fundamental Research Funds for the Central Universities} (Grant No. 2016-34000-31610258) and \emph{The Young Teachers' Science Research Funds of Anhui University of Technology} (Grant No. RD16100248 ). }

\begin{abstract}
We show a closed Bach-flat Riemannian manifold with a fixed positive constant scalar curvature has to be locally spherical if its Weyl and traceless Ricci tensors are small in the sense of either $L^\infty$ or $L^{\frac{n}{2}}$-norm. Compared with the complete non-compact case done by Kim, we apply a different method to achieve these results. These results generalize a rigidity theorem of positive Einstein manifolds due to M.-A.Singer. As an application, we can partially recover the well-known Chang-Gursky-Yang's $4$-dimensional conformal sphere theorem.
\end{abstract}

\maketitle



\section{Introduction}

The notion of Bach tensor was first introduced by Rudolf Bach in 1921 (see \cite{Bach}) when studying the so-called \emph{conformal gravity}. That is, instead of using the Hilbert-Einstein functional, one consider the functional
$$\mathcal{W}(g) = \int_{M^4} |W (g)|^2 dv_g$$
on $4$-dimensional manifolds. The corresponding critical points of this functional are characterized by the vanishing of certain symmetric $2$-tensor $B_g$. The tensor $B_g$ is usually referred as Bach tensor and the metric is called \emph{Bach-flat}, if $B_g$ vanishes. \\

Let $(M^n,g)$ be an $n$-dimensional Riemannian manifold ($n\geq 4$).  The \emph{Bach tensor} is defined to be
\begin{align}
B_{jk} = \frac{1}{n-3} \nabla^i \nabla^l W_{ijkl} + W_{ijkl}S^{il},
\end{align}
where 
\begin{align}
S_{jk} = \frac{1}{n-2} \left(R_{jk} - \frac{1}{2(n-1)} R g_{jk}\right)
\end{align}
is the Schouten tensor.\\

Using the \emph{Cotton tensor}  
\begin{align}
C_{ijk} = \nabla_i S_{jk} - \nabla_j S_{ik}
\end{align}
and the relation
\begin{align}
\nabla^l W_{ijkl} = (n-3) C_{ijk},
\end{align}
we can extend the definition of Bach tensor such that it can be defined for $3$-dimensional manifolds:
\begin{definition}
For any $n\geq 3$, the Bach tensor is defined to be
\begin{align}
B_{jk} = \nabla^i C_{ijk} + W_{ijkl}S^{il}.
\end{align}
We say a metric is \emph{Bach-flat}, if its Bach tensor vanishes.\\
\end{definition}

Typical examples of Bach flat metrics are Einstein metrics and locally conformally flat metrics. Due to the conformal invariance of Bach-flatness on $4$-manifolds, metric conformal to Einstein metrics are also Bach-flat. For $4$-dimensional manifolds, it also includes half-locally conformally flat metrics. In general, Tian and Viaclovsky studied the module space of $4$-dimensional Bach-flat manifolds (cf. \cite{T-V_1, T-V_2}). Besides these known "trivial" examples, there are not many examples known about generic Bach-flat manifolds so far. In fact, in some particular situations, one would expect rigidity phenomena occur.\\

In \cite{Kim}, Kim shows that on a complete non-compact $4$-dimensional Bach-flat manifold $(M, g)$ with zero scalar curvature and positive Yamabe constant has to be flat, if the $L^2(M,g)$-norm of its Riemann curvature tensor is sufficiently small. This result can be easily extended to any dimension $n \geq 3$. \\

Kim's proof is based on a classic idea that one can get global rigidity from local estimates: applying the ellipticity of Bach-flat metric, the Sobolev's inequality and together the smallness of $||Rm||_{L^2(M,g)(M, g)}$, one can get the estimate
\begin{align*}
||Rm||_{L^4(B_r(p), g)} \leq \frac{C}{r}||Rm||_{L^2(M,g)(M, g)}
\end{align*}
for any fixed $p\in M$ and $r > 0$. Now the conclusion follows by letting $r \rightarrow \infty$.\\

This method can also be used in various problems, for example, see \cite{Chen}. However, note that the assumption of non-compactness is essential here. One can not get the rigidity by simply letting $r \rightarrow \infty$, when the manifold is compact without boundary for instance.\\

Is it possible for us to have a result similar to Kim's but on closed manifolds? Here by \emph{closed manifolds}, we mean compact manifolds without boundary. In fact, Singer proved that even dimensional closed positive Einstein manifolds with non-vanishing Euler characteristic have to be locally spherical, provided the $L^{\frac{n}{2}}$-norm of its Weyl tensor is small (cf. \cite{Singer}). As a special case of Bach-flat metric, this result suggests that this phenomenon might occur in a larger class. \\

Applying a global estimate for symmetric $2$-tensors (see Proposition \ref{prop:ineq_symmetric_2_tensor_est}), we can prove the following result:
\newtheorem*{thm_A}{\bf Theorem A}
\begin{thm_A}\label{thm:sphere_thm_Bach_flat_L^infty}
Suppose $(M^n, g)$ is a closed Bach-flat Riemannian manifold with constant scalar curvature
$$R_g = n(n-1).$$
If
\begin{align}
||W||_{L^\infty(M, g)} + ||E||_{L^\infty(M, g)} < \varepsilon_0 (n):=\frac{n-1}{4}, 
\end{align}
then $(M, g)$ is isometric to a quotient of the round sphere $\mathbb{S}^n$.
\end{thm_A}

\begin{remark}
Note that in Theorem A we do not assume the Yamabe constant is uniformly positively lower bounded. This assumption will be needed in Theorem B. It is equivalent to the existence of a uniform Sobolev's inequality (see section 4), which was applied frequently in the proof of Theorem B.
\end{remark}

Another one by assuming integral conditions:
\newtheorem*{thm_B}{\bf Theorem B}
\begin{thm_B}\label{thm:sphere_thm_Bach_flat_L^{n/2}}
Suppose $(M^n, g)$ is a closed Bach-flat Riemannian manifold with constant scalar curvature
$$R_g = n(n-1).$$ Assume that there is a constant $\alpha_0$ such that its Yamabe constant satisfies that
\begin{align}
Y(M, [g]) \geq \alpha_0 > 0.
\end{align}
Then $(M,g)$ is isometric to a quotient of the round sphere $\mathbb{S}^n$, if
\begin{align}\label{presumption:tau_0}
||W||_{L^{\frac{n}{2}}(M,g)} + ||E||_{L^{\frac{n}{2}}(M,g)} <  \tau_0 (n, \alpha_0):= \frac{3\alpha_0}{32n(n-1)}.
\end{align}
\end{thm_B}

\begin{remark}
Bach-flat metrics is one of the typical examples of the so-called \emph{critical metrics} (cf. \cite{T-V_1}). By replacing the presumption \emph{Bach-flatness} with \emph{harmonic curvature}, which refers to the vanishing of Cotton tensor when the scalar curvature is a constant, the corresponding version of Theorem A and B are still valid without any essential difficulty.
\end{remark}

In particular, applying Theorem B for $4$-dimensional manifolds, we can partially recover the well-known $4$-dimensional conformal sphere theorem by Chang-Gursky-Yang (cf. \cite{C-G-Y}; for a generalization see \cite{C-Z}):
\newtheorem*{thm_C}{\bf Theorem C}
\begin{thm_C}\label{thm:sphere_thm_Bach_flat_L^{n/2}}
Suppose $(M^4, g)$ is a closed Bach-flat Riemannian manifold. Assume that there is a constant $\alpha_0$ such that its Yamabe constant satisfies that
\begin{align}
Y(M, [g]) \geq \alpha_0 > 0.
\end{align}
Then $(M,g)$ is conformal to the round sphere $\mathbb{S}^4$ or its canonical quotient $\mathbb{R}P^4$, if
\begin{align}
\int_{M^4} |W_g|^2 dv_g  < \frac{32}{3} \pi^2 (\chi(M^4) - 2) +  \frac{\alpha_0}{192}.
\end{align}
\end{thm_C}

\begin{remark}
It was shown in \cite{C-G-Y} that $(M^4, g)$ is conformal to $(\mathbb{C}P^2, g_{FS})$ or a manifold covered isometrically by $S^1\times S^3$ endowed with the canonical product metric, if we assume
\begin{align}
\int_{M^4} |W_g|^2 dv_g  =  16 \pi^2 \chi(M^4) 
\end{align}
instead.
\end{remark}

\paragraph{\textbf{Acknowledgement}}

The author would like to express their appreciations to Professor Huang Xian-Tao for his interests in this problem and inspiring discussions.\\


\section{$\theta$-Codazzi tensor and related inequality}
We define a concept which generalizes the classic \emph{Codazzi tensor}:
\begin{definition}
For any $\theta \in \mathbb{R}$, we say a symmetric $2$-tensor $h \in S_2(M)$ is a $\theta$-Condazzi tensor if
\begin{align}
C_\theta (h)_{ijk} := \nabla_i h_{jk} - \theta \nabla_j h_{ik} = 0.
\end{align}
In particular, $h$ is referred to be a \emph{Codazzi tensor} or \emph{anti-Codazzi tensor} if $\theta = 1$ or $\theta = -1$ respectively.
\end{definition}

The motivation for us to define this notion is the following identity associated to it:
\begin{lemma}\label{lem:theta_Codazzi_identiy}
Suppose $(M, g)$ is a closed Riemannian manifold with constant scalar curvature
$$R_g = n(n-1)\lambda.$$
Then for any $h \in S_2(M)$ and $\theta \in \mathbb{R}$,
\begin{align}
 &\int_M \left(|\nabla h|^2 - \frac{1} {1 + \theta^2} |C_\theta (h)|^2 \right) dv_{g} 
 \\
 =&  \frac{2 \theta} {1 + \theta^2}\int_M \left[ |\delta h|^2 + W(\overset{\circ}{h}, \overset{\circ}{h}) + \frac{2}{n-2}(tr h) E\cdot h - \frac{n}{n-2} tr (E \times h^2) - n \lambda |\overset{\circ}{h}|^2  \right] dv_{g} ,\notag
\end{align}
where $\overset{\circ}{h} := h - \frac{1}{n}(tr h) g$ is the traceless part of the tensor $h$.
\end{lemma}

\begin{proof}
We have 
\begin{align*}
&\int_M \nabla_i h_{jk} \nabla^j h^{ik} dv_{g}\\
=& -\int_M  \nabla_j\nabla_i h_k^{\ j}  h^{ik} dv_{g}\\
=&-\int_M ( \nabla_i \nabla_j h_k^{\ j} + R_{j i l}^j h^{\ l}_k - R_{j i k}^l h_l^{\ j})  h^{i k} dv_{g}\\
=&-\int_M ( - \nabla_i (\delta h)_k + R_{i l} h^{\ l}_k - R_{j i k l} h^{j l})  h^{ik} dv_{g}\\
=&\int_M \left[|\delta h|^2 - \left( E_{i l} h^{\ l}_k + (n-1) \lambda g_{il} h^{\ l}_k \right) h^{ik}\right] dv_{g} \\
&+ \int_M \left( W_{jikl} + \frac{2}{n-2}(E_{jl}g_{ik} - E_{jk}g_{il}) + \lambda ( g_{jl} g_{ik} - g_{jk } g_{il} ) \right) h^{jl}h^{ik} dv_{g}\\
=& \int_M \left[|\delta h|^2 + W (h, h) + \lambda ( (tr h)^2 - n |h|^2 ) + \frac{2}{n-2}(tr h) E\cdot h - \frac{n}{n-2} tr (E \times h^2)\right] dv_{g}\\
=& \int_M \left[|\delta h|^2 + W(\overset{\circ}{h}, \overset{\circ}{h}) + \frac{2}{n-2}(tr h) E\cdot h - \frac{n}{n-2} tr (E \times h^2) - n \lambda |\overset{\circ}{h}|^2 \right] dv_{g}.
\end{align*}
Thus for any $\theta \in \mathbb{R}$,
\begin{align*}
& \int_M |C_\theta (h)|^2 dv_g\\
=& \int_M |\nabla_i h_{jk} - \theta \nabla_j h_{ik}|^2 dv_{g}\\
=& \int_M \left[ (1 + \theta^2)|\nabla h|^2 - 2 \theta \nabla_i h_{jk}\nabla^j h^{ik}\right]  dv_{g}\\
=& \int_M \left[(1 + \theta^2)|\nabla h|^2 - 2 \theta  \left(|\delta h|^2 + W(\overset{\circ}{h}, \overset{\circ}{h}) + \frac{2}{n-2}(tr h) E\cdot h - \frac{n}{n-2} tr (E \times h^2) - n \lambda |\overset{\circ}{h}|^2  \right) \right] dv_{g}.
\end{align*}
That is,
\begin{align*}
 &\int_M \left(|\nabla h|^2 - \frac{1} {1 + \theta^2} |C_\theta (h)|^2 \right) dv_{g} 
 \\
 =&  \frac{2 \theta} {1 + \theta^2}\int_M \left[ |\delta h|^2 + W(\overset{\circ}{h}, \overset{\circ}{h}) + \frac{2}{n-2}(tr h) E\cdot h - \frac{n}{n-2} tr (E \times h^2) - n \lambda |\overset{\circ}{h}|^2  \right] dv_{g}  .
\end{align*}
\end{proof}

From this, we get the following inequality:
\begin{proposition}\label{prop:ineq_symmetric_2_tensor_est}
Suppose $(M, g)$ is a closed Riemannian manifold with constant scalar curvature
$$R_g = n(n-1)\lambda.$$
Then for any $h \in S_2(M)$ and $\theta \in \mathbb{R}$,
\begin{align}
 \int_M |\nabla h|^2  dv_{g} 
 \geq & \frac{2 \theta} {1 + \theta^2}\int_M \left[ |\delta h|^2 + W(\overset{\circ}{h}, \overset{\circ}{h} ) + \frac{2}{n-2}(tr h) E\cdot h- \frac{n}{n-2} tr (E \times h^2)  - n\lambda|\overset{\circ}{h}|^2 \right]dv_{g},
\end{align}
where equality holds if and only $h$ is a $\theta$-Codazzi tensor.
\end{proposition}

In particular, we have
\begin{corollary}\label{cor:sym_tensor_est}
Suppose $(M, g)$ is a closed Riemannian manifold with constant scalar curvature
$$R_g = n(n-1)\lambda.$$
Then the traceless part of Ricci tensor satisfies 
\begin{align}
 \int_M |\nabla E|^2  dv_{g} 
 \geq & \frac{2 \theta} {1 + \theta^2}\int_M \left[  W(E, E ) - \frac{n}{n-2} tr E^3  - n\lambda|E|^2 \right]dv_{g},
\end{align}
In particular when $\theta = 1$, the equality holds if and only if $g$ is of harmonic curvature.
\end{corollary}

\begin{proof}
By the second Bianchi identity, we can easily see that 
$$\delta E = - \frac{n-2}{2n} d R_g = 0.$$
Note that $tr E = 0$,  thus the conclusion follows.

When $\theta = 1$, $E$ is a Codazzi tensor if and only if the Cotton tensor vanishes:
$$C_{ijk} = \frac{1}{n-2}\underset{i,j}{Alt}\left( \nabla_i R_{jk} - \frac{1}{2(n-1)}  g_{jk}\nabla_i R\right) = 0.$$
\end{proof}


\section{$L^\infty$-sphere theorem}

We can rewrite the Bach tensor in terms of traceless Ricci tensor:
\begin{lemma}\label{lem:Bach_expression}
The Bach tensor can be expressed as follow
\begin{align}
B_g =& \frac{1}{n-2}\Delta_g E - \frac{1}{2(n-1)} \left( \nabla^2_g R - \frac{1}{n} g \Delta_g R\right) +\frac{ 2}{n-2} \overset{\circ}{W} \cdot E \\
\notag &- \frac{n}{ (n-2)^2} \left( E\times E - \frac{1}{n}|E|^2 g  \right)  - \frac{R}{(n-1)(n-2)} E,
\end{align}
where $ (\overset{\circ}{W} \cdot E)_{jk} := W_{ijkl}E^{il}$.
\end{lemma}

\begin{proof}
By definition,
\begin{align*}
\nabla^i C_{ijk} &= \nabla^i ( \nabla_i S_{jk} - \nabla_j S_{ik} )\\
&= \Delta_g S_{jk} - ( \nabla_j \nabla_i S^i_k + R_{ijp}^i S_k^p - R_{ijk}^p S_p^i )\\
&= \Delta_g S_{jk} - \nabla_j \nabla_k tr S -  (Ric \times S)_{jk} + (\overset{\circ}{Rm} \cdot S)_{jk},
\end{align*}
where we used the fact
$$\nabla_i S^i_k = \nabla_k tr S$$
by the contracted second Bianchi identity.

Since
$$S = \frac{1}{n-2}E + \frac{R}{2n(n-1)} g$$
and
$$Rm = W + \frac{1}{n-2} E \owedge g + \frac{R}{2n(n-1)}g \owedge g,$$
the conclusion follows by substituting them into 
$$B_{jk} = \nabla^i C_{ijk} + W_{ijkl}S^{il}.$$
\end{proof}

As the first step, we show the metric has to be Einstein under given presumptions:
\begin{proposition}\label{prop:Bach_flat_Einstein}
For $n\geq 3$, there exists a constant $\Lambda_n > 0$ only depends on $n$, such that any closed Bach flat Riemannian manifold $(M^n, g)$ with constant scalar curvature
$$R_g = n(n-1)$$ and 
$$||W_g||_{L^\infty(M, g)}+ ||E_g||_{L^\infty(M, g)} < \Lambda_n:= \frac{n}{3}$$
has to be Einstein.
\end{proposition}

\begin{proof}
Since the scalar curvature $R_g$ is a constant, by Lemma \ref{lem:Bach_expression},
\begin{align*}
B_g = \frac{1}{n-2}\Delta_g E  +\frac{ 2}{n-2} \overset{\circ}{W} \cdot E - \frac{n}{ (n-2)^2} \left( E\times E - \frac{1}{n}|E|^2 g  \right)  - \frac{n}{n-2} E = 0.
\end{align*}
That is,
\begin{align*}
\Delta_g E  +2\overset{\circ}{W} \cdot E - \frac{n}{ n-2} \left( E\times E - \frac{1}{n}|E|^2 g  \right)  - n E = 0.
\end{align*}
Thus,
\begin{align*}
-E\Delta_g E = 2 W (E, E) - \frac{n}{ n-2} tr( E^3 )  - n|E|^2
\end{align*}
and hence
\begin{align}\label{eqn:int_nabla_E}
\int_M |\nabla E|^2 dv_g = -\int_M E\Delta_g E dv_g = \int_M \left(2 W (E, E) - \frac{n}{ n-2} tr( E^3 )  - n |E|^2 \right) dv_g.
\end{align}

On the other hand, from Corollary \ref{cor:sym_tensor_est},
\begin{align*}
 \int_M |\nabla E|^2  dv_{g} 
 \geq & \frac{2 \theta} {1 + \theta^2}\int_M \left(  W(E, E ) - \frac{n}{n-2} tr E^3  - n|E|^2 \right)dv_{g},
\end{align*}
for any $\theta \in \mathbb{R}$.

Therefore,
\begin{align}\label{ineq:W(E,E)_|E|^2}
\frac{2(1 - \theta + \theta^2 )}{(1- \theta)^2} \int_M  W(E, E ) dv_g \geq  \frac{n  }{n-2}\int_M \left( tr E^3  + (n-2)|E|^2 \right)dv_{g}.
\end{align}

Since
$$
\int_M  W(E, E ) dv_g \leq ||W||_{L^\infty(M, g)}\int_M |E|^2 dv_g, 
$$
by taking $\theta = -1$, we get
$$
\frac{n  }{n-2}\int_M \left( tr E^3  + (n-2)|E|^2 \right)dv_{g} \leq \frac{3}{2}||W||_{L^\infty(M, g)}\int_M |E|^2 dv_g.
$$
That is,
$$
\frac{n  }{n-2}\int_M  tr E^3  dv_{g} \leq \left(\frac{3}{2}||W||_{L^\infty(M, g)} - n \right)\int_M |E|^2 dv_g.
$$

From the inequality 
$$\int_M  tr E^3  dv_{g} \geq - \int_M |E|^3 dv_g \geq - ||E||_{L^\infty(M, g)}\int_M |E|^2 dv_g ,$$
we have
$$\left(\frac{3}{2}||W||_{L^\infty(M, g)} + \frac{n}{n-2}||E||_{L^\infty(M, g)} - n\right)\int_M |E|^2 dv_g \geq 0.$$
Therefore for any metric $g$ satisfies
$$||W||_{L^\infty(M, g)} + ||E||_{L^\infty(M, g)} < \Lambda_n:= \frac{n}{3},$$
we have $E = 0$.
\end{proof}

It is well-known that the Weyl tensor satisfies an elliptic equation on Einstein manifolds (cf. \cite{Singer}):
\begin{lemma}\label{lem:Weyl_eqn_Einstein}
Let $(M^n,g)$ be an Einstein manifold with scalar curvature
$$R_g = n(n-1) \lambda,$$
then its Weyl tensor satisfies 
\begin{align}\label{Weyl}
\Delta_g W - 2(n-1)\lambda W - 2 \mathcal{Q} (W) = 0,
\end{align}
where $\mathcal{Q} (W) := B_{ijkl} - B_{jikl} + 
B_{ikjl} -  B_{jkil}$ is a quadratic combination of Weyl
tensors with $B_{ijkl} := g^{pq}g^{rs} W_{pijr} W_{qkls}$.\\
\end{lemma}

Now we finish this section by proving one of our main theorem:
\begin{proof}[Proof of Theorem A]
We take
$$\varepsilon_0 : =\min \{ \Lambda_n, \frac{n-1}{4} \} = \frac{n-1}{4}.$$
From Proposition \ref{prop:Bach_flat_Einstein}, we conclude that $g$ is an Einstein metric. Applying Lemma \ref{lem:Weyl_eqn_Einstein}, we have
\begin{align*}
- \int_M \langle \Delta_g W - 2(n-1) W, W \rangle dv_g = -2 \int_M \langle \mathcal{Q}(W), W \rangle dv_g \leq 8 \int_M |W|^3 dv_g.
\end{align*}
That is,
\begin{align}\label{ineq:int_Weyl}
\int_M \left( |\nabla W|^2 + 2(n-1) |W|^2 \right) dv_g  \leq 8 \int_M |W|^3 dv_g.
\end{align}

Now we have
\begin{align*}
2(n-1)\int_M  |W|^2 dv_g  \leq 8 \int_M |W|^3 dv_g \leq 8 ||W||_{L^\infty(M, g)} \int_M |W|^2 dv_g.
\end{align*}
Thus the Weyl tensor vanishes, since $$||W||_{L^\infty(M, g)} < \varepsilon_0 =\frac{n-1}{4}.$$
Therefore, the metric $g$ is locally spherical.
\end{proof}

\section{$L^{\frac{n}{2}}$-sphere theorem}

Let $(M, g)$ be an Riemannian manifold. Suppose the Yamabe constant associated to it satisfies that
$$
Y(M, [g]) := \inf_{0 \not\equiv u \in C^\infty(M)} \frac{\int_M \left(\frac{4(n-1)}{n-2}|\nabla u|^2 + R_g u^2 \right)dv_g}{\left(\int_M u^{\frac{2n}{n-2}} dv_g \right)^{\frac{n-2}{n}}} \geq \alpha_0 > 0.
$$
By normalizing the scalar curvature such that $R_g = n(n-1)$, we get
\begin{align*}
\left(\int_M u^{\frac{2n}{n-2}} dv_g \right)^{\frac{n-2}{n}} \leq& \frac{1}{Y(M, [g])} \int_M \left(\frac{4(n-1)}{n-2}|\nabla u|^2 + R_g u^2 \right)dv_g \\
=& \frac{n(n-1)}{Y(M, [g])} \int_M \left(\frac{4}{n(n-2)}|\nabla u|^2 + u^2 \right)dv_g \\
\leq& \frac{4n(n-1)}{3Y(M, [g])} \int_M \left(|\nabla u|^2 + u^2 \right)dv_g \\
\leq& \frac{4n(n-1)}{3\alpha_0} \int_M \left(|\nabla u|^2 + u^2 \right)dv_g \\
\end{align*}
Denote $C_S:= \frac{4n(n-1)}{3\alpha_0} > 0$, we get the \emph{Sobolev's inequality}
\begin{align}\label{ineq:Sobolev}
\left(\int_M u^{\frac{2n}{n-2}} dv_g \right)^{\frac{n-2}{n}} \leq C_S \int_M \left(|\nabla u|^2 + u^2 \right)dv_g
\end{align}
Note that, the constant $C_S > 0$ only depends on $n$ and $\alpha_0$ and is independent of the metric $g$.

\begin{lemma}\label{lem:Bach_flat_L^{n/2}_Einstein}
Let $(M^n, g)$ be a Bach flat Riemannian manifold with constant scalar curvature
$$R_g = n(n-1).$$ Suppose there is a constant $\alpha_0$ such that its Yamabe constant satisfies that
\begin{align}
Y(M, [g]) \geq \alpha_0 > 0.
\end{align}
Then $(M^n, g)$ is Einstein, if
\begin{align}
||W||_{L^{\frac{n}{2}}(M,g)} + ||E||_{L^{\frac{n}{2}}(M,g)} <  \delta_0:= \frac{\alpha_0}{4n(n-1)} = \frac{1}{3C_S}.
\end{align}
\end{lemma}

\begin{proof}
From equation (\ref{eqn:int_nabla_E}) and \emph{H\"older's inequality},
\begin{align*}
\int_M |\nabla E|^2 dv_g &= \int_M \left(2 W (E, E) - \frac{n}{ n-2} tr( E^3 )  - n |E|^2 \right) dv_g\\
&\leq \left( 2||W||_{L^{\frac{n}{2}}(M,g)} + \frac{n}{ n-2}||E||_{L^{\frac{n}{2}}(M,g)}\right) ||E||^2_{L^{\frac{2n}{n-2}}(M,g)} - n||E||^2_{L^2(M,g)} \\
&\leq 3 \delta_0 ||E||^2_{L^{\frac{2n}{n-2}}(M,g)} - n||E||^2_{L^2(M,g)}.
\end{align*}
By \emph{Sobolev's inequality} (\ref{ineq:Sobolev}) and the \emph{Kato's inequality},
$$||E||^2_{L^{\frac{2n}{n-2}}(M,g)} \leq C_S \left(||\nabla |E| ||^2_{L^2(M,g) }+ ||E||^2_{L^2(M,g)} \right) \leq C_S \left(||\nabla E ||^2_{L^2(M,g)} + ||E||^2_{L^2(M,g)} \right).
$$
Thus, we have
\begin{align*}
||\nabla E||^2_{L^2(M,g)}\leq& 3 \delta_0 C_S \left(||\nabla E ||^2_{L^2(M,g)} + ||E||^2_{L^2(M,g)} \right) - n||E||^2_{L^2(M,g)} \\
=& ||\nabla E ||^2_{L^2(M,g)} - (n - 1) ||E||^2_{L^2(M,g)}  .
\end{align*}
Therefore, $E$ vanishes identically on $M$ and hence $(M,g)$ is Einstein.
\end{proof}

Now we can show
\begin{proof}[Proof of Theorem B]
From Lemma \ref{lem:Bach_flat_L^{n/2}_Einstein}, $(M,g)$ has to be Einstein. Now from \emph{Sobolev's inequality} (\ref{ineq:Sobolev}), \emph{Kato's inequality} and inequality (\ref{ineq:int_Weyl}), we have
\begin{align*}
||W||^2_{L^{\frac{2n}{n-2}}(M,g)} &\leq C_S \int_M \left( |\nabla |W||^2 + |W|^2 \right) dv_g
\leq  C_S \int_M \left( |\nabla W|^2 + |W|^2 \right) dv_g 
\leq  8 C_S \int_M  |W|^3 dv_g .
\end{align*} 
Applying \emph{H\"older's inequality},
\begin{align*}
\int_M  |W|^3 dv_g \leq ||W||_{L^{\frac{n}{2}}(M,g)}||W||^2_{L^{\frac{2n}{n-2}}(M,g)} 
\end{align*}
and hence
$$\left( 1 - 8C_S  ||W||_{L^{\frac{n}{2}}(M,g)}\right)||W||^2_{L^{\frac{2n}{n-2}}(M,g)} \leq 0,$$
which implies that $W$ vanishes identically on $M$ since
$$||W||_{L^{\frac{n}{2}}(M,g)} < \tau_0:=\frac{3\alpha_0}{32n(n-1)} = \frac{1}{8C_S}.$$
Therefore, $(M, g)$ is isometric to a quotient of $\mathbb{S}^n$.
\end{proof}

As for $n = 4$, we have
\begin{proof}[Proof of Theorem C]
Let $\hat g \in [g]$ be the Yamabe metric, which means
$$R_{\hat g} \left(Vol(M^4, \hat g)\right)^{\frac{1}{2}} = Y(M^4, [g]).$$
We can also normalize it such that 
$$R_{\hat  g} = 12.$$
According to the solution of \emph{Yamabe problem}, 
$$Y(M^4, [g]) \leq Y(\mathbb{S}^4, g_{\mathbb{S}^4}) = 12 \cdot \left( \frac{8}{3} \pi^2\right)^{\frac{1}{2}} = 8 \sqrt{6} \pi$$
and hence
$$Vol(M^4, \hat g )\leq Vol(\mathbb{S}^4, g_{\mathbb{S}^4}) = \frac{8}{3} \pi^2.$$

From the \emph{Gauss-Bonnet-Chern formula},
\begin{align}
\int_{M^4} \left( Q_{\hat g} + \frac{1}{4} |W_{\hat g}|^2 \right) dv_{\hat g} = 8\pi^2 \chi(M^4),
\end{align}
where 
\begin{align}
Q_{\hat g}:= - \frac{1}{6} \Delta_{\hat g} R_{\hat g} - \frac{1}{2}|E_{\hat g}|^2 + \frac{1}{24} R_{\hat g}^2
\end{align}
is the \emph{Q-curvature} for metric $\hat g$. Thus,
$$||E_{\hat g}||^2_{L^2(M, \hat g)} = \frac{1}{2} ||W_{\hat g}||^2_{L^2(M, \hat g)} + 12 Vol(M^4, \hat g) - 16 \pi^2 \chi(M^4) \leq \frac{1}{2} ||W_{\hat g}||^2_{L^2(M, \hat g)}+ 16 \pi^2 ( 2 -\chi(M^4))$$
and hence
\begin{align*}
||W_{\hat g}||^2_{L^2(M, \hat g)} + ||E_{\hat g}||^2_{L^2(M, \hat g)} &\leq \frac{3}{2} ||W_{\hat g}||^2_{L^2(M, \hat g)} + 16 \pi^2 ( 2 -\chi(M^4)) \\
&= \frac{3}{2} ||W_g||^2_{L^2(M,g)} + 16 \pi^2 ( 2 -\chi(M^4))\\
& < \frac{\alpha_0}{128},
\end{align*}
where we used the fact that $||W_g||_{L^2(M,g)}$ is conformally invariant for $4$-dimensional manifolds.

On the other hand, the metric $\hat g$ is also Bach-flat, since Bach-flatness is conformally invariant for $4$-dimensional manifolds. Applying Theorem B to the Yamabe metric $\hat g$, we conclude that $(M^4, \hat g)$ is isometric to a quotient of the round sphere $\mathbb{S}^4$. 

For the quotient of an even dimensional sphere, only identity and $\mathbb{Z}_2$-actions make it a smooth manifold. Therefore, $(M^4, g)$ is conformal to $\mathbb{S}^4$ or $\mathbb{R}P^4$ with canonical metrics.
\end{proof}

\bibliographystyle{amsplain}

\end{document}